\documentclass[a4paper,10pt]{article}
\usepackage{mathtext}
\usepackage[T1,T2A]{fontenc}
\usepackage[cp1251]{inputenc}
\usepackage[english]{babel}
\usepackage{amsmath}
\usepackage{amsfonts}
\usepackage{amssymb}
\usepackage{mathrsfs}
\usepackage{amsthm}
\usepackage{titling}

\usepackage{url}
\usepackage{enumerate}

\usepackage{color}
\usepackage{euscript}

\DeclareMathOperator{\dist}{dist}

\DeclareMathOperator{\diam}{diam}

\newcommand{\Bell}{\boldsymbol{B}}

\newcommand{\Class}{\boldsymbol{A}}

\newcommand{\ClassC}{\boldsymbol{A}^{\circ}}
\newcommand{\ClassCs}{\ClassC_{\mathrm{s}}}

\newcommand{\BMO}{\mathrm{BMO}}

\newcommand{\eps}{\varepsilon}
\newcommand{\vf}{\varphi}
\newcommand{\Omm}{\mathfrak{S}} 

\newcommand{\Dell}{\mathfrak{D}} 

\DeclareMathOperator{\cl}{cl}

\DeclareMathOperator{\E}{\mathbb{E}}
\newcommand{\av}[2]{\langle {#1}\rangle_{{}_{#2}}}
\renewcommand{\le}{\leqslant}

\renewcommand{\leq}{\leqslant}
\renewcommand{\geq}{\geqslant}

\newcommand{\per}{\hbox{\tiny \textup{per}}}

\newcommand{\Om}{\mathfrak{W}}
\newcommand{\Ommuc}{\mathfrak{W}_{\mathrm{muc}}} 

\newcommand{\M}{\EuScript{M}}

\newcommand{\Mpr}{\M_{\mathrm{pr}}}

\newtheorem{Le}{Lemma}[section]
\newtheorem{Def}[Le]{Definition}

\newtheorem{Th}[Le]{Theorem}
\newtheorem{Cor}[Le]{Corollary}
\newtheorem{Rem}[Le]{Remark}

\numberwithin{equation}{section}

\begin{document}
\author{Dmitriy~Stolyarov 
\and Pavel~Zatitskiy}
\title{Sharp transference principle for~$\BMO$ and~$A_p$
\thanks{Supported by the Russian Science Foundation grant 19-71-10023.}}
\maketitle
\begin{abstract}
We provide a version of the transference principle. It says that certain optimization problems for functions on the circle, the interval, and the line have the same answers. In particular, we show that the sharp constants in the John--Nirenberg inequalities for naturally defined $\BMO$-spaces on the circle, the interval, and the line coincide. The same principle holds true for the Reverse H\"older inequality for Muckenhoupt weights. 
\end{abstract}

\section{Preliminaries and statement of the results}
\subsection{The~$\BMO$ space and the John--Nirenberg inequality}
The classical definition of the~$\BMO(\mathbb{R})$ seminorm reads as follows:
\begin{equation}\label{BMOClassicalDefintion}
\|\vf\| = \sup_{J\ \hbox{\tiny interval}}\frac{1}{|J|}\int\limits_{J} \big|\vf(x) - \av{\vf}{J}\big|\,dx, \quad \vf \in L_{1,\mathrm{loc}}(\mathbb{R}).
\end{equation}
Here~$\av{\vf}{J} = |J|^{-1}\int_J \vf(x)\,dx$ is the average of~$\vf$ over~$J$. The space of functions of bounded mean oscillation defined by this seminorm has many important properties. One of them is the quantitative bound called the John--Nirenberg inequality. It says that there exist constants~$C_1^{\mathbb{R}}$ and~$C_2^{\mathbb{R}}$ such that for any interval~$J \subset \mathbb{R}$, any~$\lambda > 0$, and any~$\vf \in \BMO(\mathbb{R})$,
\begin{equation*}
\frac{1}{|J|}\Big|\{x\in J\mid |\vf(x) - \av{\vf}{J}| \geq \lambda\}\Big| \leq C_1^{\mathbb{R}}e^{-\frac{C_2^{\mathbb{R}}\lambda}{\|\vf\|}}.
\end{equation*}
This inequality, in particular, implies that one can equip~$\BMO$ with an equivalent seminorm
\begin{equation}\label{BMOp}
\|\vf\|_{p,\mathbb{R}} = \sup_{J\ \hbox{\tiny interval}}\Big(\frac{1}{|J|}\int\limits_{J} \big|\vf(x) - \av{\vf}{J}\big|^p\,dx\Big)^{\frac{1}{p}}, \quad p \in [1,\infty).
\end{equation}
Moreover, there exists a constant~$C_3^{\mathbb{R}}$ such that
\begin{equation}\label{ShortIntJN}
\int\limits_{J}e^{\frac{C_3^{\mathbb{R}} \vf}{\|\vf\|_{1,\mathbb{R}}}} < \infty, \quad \hbox{for any }\vf \in \BMO(\mathbb{R}) \hbox{ and any interval } J.
\end{equation}
The latter inequality is called the integral form of the John--Nirenberg inequality. 
\begin{Rem}
A more classical way to state the integral form of the John--Nirenberg inequality is\textup{:} for any~$\vf \in\BMO(\mathbb{R})$
\begin{equation*}
\sup\limits_J \Big\langle\exp\Big(C_3^{\mathbb{R}}\frac{\vf - \av{\vf}{J}}{\|\vf\|_{1,\mathbb{R}}}\Big)\Big\rangle_J < \infty,
\end{equation*}
where~$J$ runs through all finite subintervals of~$\mathbb{R}$. This form is equivalent to~\eqref{ShortIntJN} with the same constant~$C_3^\mathbb{R}$. For example, this equivalence can be derived from Theorem~\textup{\ref{CoincedenceBMO}} below.
\end{Rem}

One may also define the space~$\BMO(I)$ consisting of integrable functions on the interval~$I$ in a similar manner. In such a case, the intervals~$J$ over which we compute the mean oscillation in formulas~\eqref{BMOClassicalDefintion} and~\eqref{BMOp}, lie inside~$I$. We note that the restriction~$\vf|_{I}$ lies in~$\BMO(I)$ for any~$\vf \in \BMO(\mathbb{R})$:
\begin{equation}\label{SimpleEmbedding}
\|\vf|_{I}\|_{1,I}\leq \|\vf\|_{1,\mathbb{R}}.
\end{equation}
However, there exist functions~$\psi \in \BMO(I)$ which cannot be extended to a function in~$\BMO(\mathbb{R})$ having the same norm (see~\cite{Shanin}).

One may wonder what are the best possible values of the constants~$C_1^{\mathbb{R}},C_2^{\mathbb{R}},C_3^{\mathbb{R}}$ available in the inequalities above. For the case of~$\BMO(I)$, 
the sharp constants~$C_1^I$,~$C_2^I$ and~$C_3^I$ in the inequalities
\begin{align}
\frac{1}{|I|}\Big|\{x\in I\mid |\vf(x) - \av{\vf}{I}| \geq \lambda\}\Big| \leq C_1^Ie^{-\frac{C_2^I\lambda}{\|\vf\|_{1,I}}};\\
\label{JNIntegralClassic}\int\limits_{I}e^{\frac{C_3^I \vf}{\|\vf\|_{1,I}}} < \infty, \quad \vf \in \BMO(I),
\end{align}
are known (note that these constants do not depend on the particular choice of~$I$). The best possible constant~$C_2^I$ equals~$\frac{2}{e}$ (see~\cite{Korenovskiy}),  and the best possible constant~$C_1^I$ for the case~$C_2^I = \frac{2}{e}$ equals~$\frac12 e^{\frac{4}{e}}$ (see~\cite{Lerner}). As for the constant~$C_3^I$, the optimal value does not exist, however, the supremum of all admissible constants in~\eqref{JNIntegralClassic} equals~$C_2^I$, as one can see from representing the integral in~\eqref{JNIntegralClassic} in terms of the distribution function of~$\vf$. We formulate this principle as a remark.
\begin{Rem}
The best possible constants~$C_{2,p}^I$ and~$C_{2,p}^{\mathbb{R}}$ in the inequalities
\begin{align*}
\frac{1}{|I|}\Big|\{x\in I\mid |\vf(x) - \av{\vf}{I}| \geq \lambda\}\Big| \leq C_{1,p}^Ie^{-\frac{C_{2,p}^I\lambda}{\|\vf\|_{p,I}}};\\
\frac{1}{|I|}\Big|\{x\in I\mid |\vf(x) - \av{\vf}{I}| \geq \lambda\}\Big| \leq C_{1,p}^{\mathbb{R}}e^{-\frac{C_{2,p}^{\mathbb{R}}\lambda}{\|\vf\|_{p,\mathbb{R}}}}
\end{align*} 
are equal to the constants~$C_{3,p}^I$ and~$C_{3,p}^{\mathbb{R}}$ respectively\textup, where the latter pair of constants is defined by the formulas
\begin{equation}\label{C3pI}
\begin{aligned}
C_{3,p}^{I} = \sup\Big\{C\;\Big|\,\forall \vf \in\BMO(I) \quad \int\limits_{I}e^{\frac{C\vf}{\|\vf\|_{p,I}}} < \infty\Big\};\\
C_{3,p}^{\mathbb{R}} = \sup\Big\{C\;\Big|\,\forall \vf \in\BMO(\mathbb{R})\ \forall J \quad \int\limits_J e^{\frac{C\vf}{\|\vf\|_{p,\mathbb{R}}}} < \infty\Big\}.
\end{aligned}
\end{equation}
\end{Rem}

One may also wonder what are the best possible constants~$C_{3,p}^\mathbb{R}$ and~$C_{3,p}^{I}$. The constant~$C_{3,p}^{I}$ was found in~\cite{Slavin} and~\cite{SlavinVasyunin}:
\begin{equation*}
C_{3,p}^{I}=
\Big(\frac{p}{e}\Big(\Gamma(p) - \int\limits_0^1 t^{p-1}e^t\Big)+1\Big)^{\frac{1}{p}}.
\end{equation*}
\begin{Th}\label{CoincedenceBMO}
For any~$p \in [1,\infty),$~$C_{3,p}^\mathbb{R} = C_{3,p}^{I}$. In particular\textup,~$C_2^{\mathbb{R}} = C_2^{I}$.
\end{Th}
The most natural reason for Theorem~\ref{CoincedenceBMO} to hold would be the possibility to extend a function in~$\BMO_p(I)$ to a function in~$\BMO_p(\mathbb{R})$ with the same or almost the same norm. Unfortunately, this cannot be true. One may pick the logarithmic function~$\log x, x \geq 0$, and show that any of its extensions~$\psi$ to the entire line satisfies
\begin{equation*}
\|\psi\|_{1,\mathbb{R}} \geq 1.001 \|\log x\|_{1,[0,1]}.
\end{equation*}
For example, one can go through the proof of Proposition~$2.1$ in~\cite{Shanin} and recover this result. In fact, the reasoning works for the case of a general~$\BMO_p$ norm without any modifications.


One may also wonder what is the sharp dependence between the norms~$\BMO_p$ and~$\BMO_q$ for different~$p$ and~$q$. The progress in this direction is mostly related to the case~$q=2$. Namely, the inequalities
\begin{align}
\label{EquivalentNormsP<2}\|\vf\|_{p,I}\leq \|\vf\|_{2,I}, \quad &1\leq p \leq 2;\\
\label{EquivalentNormsP>2}\|\vf\|_{2,I}\leq \|\vf\|_{p,I}\leq \Big(\frac{p}{2}\Gamma(p)\Big)^{\frac{1}{p}}\|\vf\|_{2,I}, \quad &p > 2
\end{align}
are sharp, as it is proved in~\cite{SlavinVasyuninLp}. 
\begin{Th}\label{LpNormsTheorem}
The same inequalities
\begin{align*}
\|\vf\|_{p,\mathbb{R}}\leq \|\vf\|_{2,\mathbb{R}}, \quad &1\leq p \leq 2;\\
\|\vf\|_{2,\mathbb{R}}\leq \|\vf\|_{p,\mathbb{R}}\leq \Big(\frac{p}{2}\Gamma(p)\Big)^{\frac{1}{p}}\|\vf\|_{2,\mathbb{R}}, \quad &p > 2
\end{align*}
hold for functions on the line. These inequalities are also sharp.
\end{Th}
\begin{Rem}
The sharp inequality
\begin{equation*}
\frac{p}{2}\Gamma(p)\|\vf\|_{2,[0,1]}^{p-2}\av{|\varphi - \av{\varphi}{[0,1]}|^2}{[0,1]}\leq \av{|\varphi - \av{\varphi}{[0,1]}|^p}{[0,1]}, \quad 1\leq p \leq 2,
\end{equation*}
was obtained in~\textup{\cite{SlavinVasyuninLp}}. Using the same methods as in the proof of Theorem~\ref{LpNormsTheorem}\textup, one may prove analogous inequality and its sharpness for~$\BMO_2(\mathbb{R})$ functions as well.
\end{Rem}

One can also formulate an analog of the John--Nirenberg inequality for the~$\BMO_2$-norm. The sharp constants in this inequality were found in~\cite{Vasyunin} and~\cite{VasyuninVolberg}:
\begin{equation}\label{JohnNirenbergQuadraticNorm}
\frac{1}{|I|}\Big|\{x\in I\mid |\vf(x) - \av{\vf}{I}| \geq \lambda\}\Big| \leq 
\begin{cases} 
1,\quad &\lambda \leq \|\vf\|_{2,I};\\
\frac{\|\vf\|_{2,I}^2}{\lambda^2},\quad &\|\vf\|_{2,I}\leq\lambda \leq 2\|\vf\|_{2,I};\\
\frac{e^2}{4}e^{-\frac{\lambda}{\|\vf\|_{2,I}}}, \quad &2\|\vf\|_{2,I}\leq \lambda,
\end{cases}
\end{equation}
and this estimate is sharp in each of the cases.
\begin{Th}\label{WeakTypeTheorem}
The inequality~\eqref{JohnNirenbergQuadraticNorm} holds true with~$\BMO_2(I)$ norm replaced with~$\BMO_2(\mathbb{R})$ norm and is sharp in each of the cases.
\end{Th}

\begin{Rem}
The inequalities~\eqref{EquivalentNormsP<2}\textup,~\eqref{EquivalentNormsP>2}\textup, and~\eqref{JohnNirenbergQuadraticNorm} are true for functions on the line\textup: they are simple consequences of~\eqref{SimpleEmbedding} and the corresponding inequalities for functions on the interval. The non-trivial part of Theorems~\ref{LpNormsTheorem} and~\ref{WeakTypeTheorem} is the sharpness of the said inequalities for functions on the line.
\end{Rem}

\begin{Rem}\label{TwoConstants}
We signalize that the sharpness of the inequality~\eqref{EquivalentNormsP<2}\textup, the first inequality in~\eqref{EquivalentNormsP>2}\textup, and the first estimate in~\eqref{JohnNirenbergQuadraticNorm}\textup, may be obtained elementary. Indeed\textup, if we choose~$I = [-1,1]$ and~$\vf = \chi_{[0,1]} - \chi_{[-1,0]}$\textup, then all these inequalities will turn into equalities. Now we need to construct a similar function~$\vf$ on the line. One may take~$\vf = \chi_{\mathbb{R}_+} - \chi_{\mathbb{R}_-}$. 
\end{Rem}

The~$\BMO$ functions are closely related to~$A_p$ weights. We survey some sharp inequalities for them in the next section.

\subsection{Reverse H\"older inequality}
Let~$p \in (1,\infty)$. The~$A_p$ constant of the weight~$w\in L_{1,\mathrm{loc}}(\mathbb{R})$,~$w> 0$, is defined by the formula
\begin{equation}\label{Apdefinition}
[w]_{A_p(\mathbb{R})} = \sup_{J\ \hbox{\tiny interval}} \av{w}{J}\av{w^{-\frac{1}{p-1}}}{J}^{p-1}.
\end{equation}
One can also define the~$A_{\infty}$ constant by passing to the limit:
\begin{equation}\label{Ainftydefinition}
[w]_{A_\infty(\mathbb{R})} = \sup_{J\ \hbox{\tiny interval}} \av{w}{J}e^{-\av{\log w}{J}}.
\end{equation}
The weights whose~$A_p$ constant is finite are called Muckenhoupt weights. These weights satisfy a self-improvement property called the Reverse H\"older inequality:
\begin{equation}\label{eq200701}
\forall p,C\quad \exists \delta = \delta(p,C),c = c(p,C,\delta) \ \hbox{ such that }\  [w]_{A_p(\mathbb{R})} \leq C \quad \Longrightarrow \quad \forall J\ \ \av{w^q}{J}^{\frac{1}{q}} < c\av{w}{J}, \quad q \in [1,1+\delta).
\end{equation}
Similar to the~$\BMO$ case, one can define the classes~$A_p(I)$ and~$A_\infty(I)$ of weights on an interval~$I$ by restricting the supremum in formulas~\eqref{Apdefinition} and~\eqref{Ainftydefinition} to the set of intervals~$J\subset I$. In this case, the sharp values for~$\delta$ and sharp estimates for~$c$ in~\eqref{eq200701} were established in~\cite{VasyuninMuckenhoupt} and~\cite{Vasyunin3}. The formulas for the constants are quite large, so we refer the reader to the original papers.
\begin{Th}\label{MuckenhouptTheorem}
For any~$p,q$\textup, and~$C$\textup, 
\begin{equation}\label{eq210702}
\sup_{[w]_{A_p(\mathbb{R})} \leq C}\frac{\av{w^q}{I}^{\frac{1}{q}}}{\av{w}{I}} = 
\sup_{[w]_{A_p(I)} \leq C}\frac{\av{w^q}{I}^{\frac{1}{q}}}{\av{w}{I}}.
\end{equation}
\end{Th}
Similar to the~$\BMO$ case, any Muckenhoupt weight on~$\mathbb{R}$ might be restricted to an interval giving the weight with smaller or the same constant:
\begin{equation}\label{TrivialEmbedding2}
[w|_{I}]_{A_p(I)} \leq [w]_{A_p(\mathbb{R})}.
\end{equation} 
In this case it is also impossible to extend an arbitrary~$A_p(I)$ weight to a weight on~$\mathbb{R}$ having the same or almost the same~$A_p$ constant.

\subsection{Functions on the circle}
We will derive Theorems~\ref{CoincedenceBMO},~\ref{LpNormsTheorem},~\ref{WeakTypeTheorem}, and~\ref{MuckenhouptTheorem} from a more general principle. Though these three theorems claim the coincidence of certain inequalities for functions on an interval and on the line, the general principle will rather work with functions on the interval and on the circle. A function on the circle might be thought of as a periodic function on the line. More precisely, if~$\vf$ is a function on the circle~$\mathbb{T}$, let~$\vf_{\per}$ be its periodic realization:
\begin{equation*}
\vf_{\per}(x) = \vf(e^{2\pi i x}), \quad x\in\mathbb{R}. 
\end{equation*}
Define the classes of functions on the circle by the rule:
\begin{align*}
[w]_{A_p(\mathbb{T})} = [w_{\per}]_{A_p(\mathbb{R})};\\
\|\vf\|_{p,\mathbb{T}} = \|\vf_{\per}\|_{p,\mathbb{R}}.
\end{align*}
The class~$A_p(\mathbb{T})$ might be thought of as a subclass of~$A_p(\mathbb{R})$.
Which, in its turn, might be thought of as a subclass of~$A_p([0,1])$ by~\eqref{TrivialEmbedding2}.
Thus, the sharp constants in various forms of the John--Nirenberg and the Reverse H\"older inequalities get better or remain the same when we pass from functions on the interval to functions on the line, and then to functions on the circle. Our main result is that these constants remain the same. We postpone its formulation to the following section.

We warn the reader that our definition of~$\BMO(\mathbb{T})$ and~$A_p(\mathbb{T})$ differ from the classical. Usually, one takes supremum of oscillations over geometric arcs (arcs that cover only a part of the circle). This usual definition fits the general approach to~$\BMO$ on manifolds (see~\cite{BrezisNirenberg}). We allow ``long arcs'' that can cover the circle several times. This makes the class of functions narrower. Thus, the constants in the inequalities we consider are the same for the case of classically defined~$\BMO$ and~$A_p$ {on $\mathbb{T}$}.

It is a common practice in analysis to transfer various statements from the circle to the line and vice versa (more general, from the torus to the Euclidean space of the same dimension). The classical transference principle for Fourier multipliers may be found in~\cite{Grafakos}. Since we claim that the John--Nirenberg inequalities and the Reverse H\"older inequalities are the same for functions on the line and the circle, it is natural to seek for a form of the transference principle that will explain this coincidence. The answer is not as straightforward as we wish. The transference principle exists, however, we will not transfer the functions or weights themselves
.

Our approach is based on two main ideas. The first one is the representation of~$\BMO(I)$ functions,~$A_p(I)$ weights, and more general objects, as terminal distributions of specific vector-valued martingales. The idea can be traced to~\cite{Vasyunin} and~\cite{VasyuninMuckenhoupt}. We will rely upon more general and modern results and definitions of~\cite{SZ}. The second idea is that for any of these specific martingales, one may construct a function in the corresponding class on the circle, whose distribution almost equals the terminal distribution of the martingale. This construction is based on appropriate rescaling of the line.

We will start with the second idea in Section~\ref{S2}. Lemma~\ref{Gluing} is the core of the matter. In Section~\ref{S3}, we will show how to derive the results for the~$\BMO_2$ and~$A_p$ cases from the general results of Section~\ref{S2} and the theory from~\cite{SZ}. Finally, Section~\ref{S4} contains the treatment of the~$\BMO_p$ case. Here we do not have a satisfactory theory and simply construct the needed martingale. A technical (and seemingly, known to the experts) statement that truncations do not increase the~$\BMO_p$ norm is needed to justify this construction. We state and prove it in the appendix.

The authors would like to thank Fedor Nazarov for communicating the main idea (using the homogenization operation in this context) to them and Leonid Slavin for attracting their attention to these questions and for exposition advice.

\section{Constructing functions from martingales}\label{S2}
In this section, we will introduce a definition of a function class that includes both $\BMO$ and Muckenhoupt weights. Similar definitions were earlier used in~\cite{CR2},~\cite{SVZ}, and~\cite{SZ}. The classes we define below are more general (we work with~$\BMO_p$ as well). 

Let~$Y \subset \mathbb{R}^d$ be a closed set, let~$X$ be a measurable space equipped with a~$\sigma$-finite measure~$\mu$. If~$\mu$ is a probability measure, then for any measurable function~$\vf \colon X \to Y$ we define its distribution~$\mu_\vf$ (we think of~$\vf$ as of a random variable):
\begin{equation*}
\mu_\vf(A) = \mu(\vf^{-1}(A)), \quad A \subset Y, \ A \text{ is Borel}.
\end{equation*}
Note that in this case~$\mu_\vf$ is a Borel probability measure on~$Y$. 

Now let~$B \subset X$ be a measurable set, $0<\mu(B)<\infty$. We can make it a probability space by normalizing measure~$\mu$ to~$\frac{\mu}{\mu(B)}$ and restricting it on $B$. This allows us to treat~$\vf|_{B}$ as a random variable and work with its distribution~$\mu_{\vf|_B}$.

If we look at formulas~\eqref{BMOp},~\eqref{Apdefinition}, and~\eqref{Ainftydefinition} we see that both the~$\BMO$ norm and the~$A_p$ constant may be expressed in terms of measures~$\mu_{\vf|_{J}}$ (or~$\mu_{w|_J}$), where~$J$ runs through a certain family of intervals (that depends whether we define our class of functions on~$I$,~$\mathbb{R}$, or~$\mathbb{T}$), and $\mu$ is the Lebesgue measure. 

Consider the space~$\M(Y)$ consisting of all signed measures of bounded variation on~$Y$. This space is a subspace of the Banach dual to the space of all continuous bounded functions on~$Y$. So we equip~$\M(Y)$ with the weak-* topology. The space~$\M(Y)$ contains the set~$\Mpr(Y)$ of all probability measures. 
We also note that~$\Mpr(Y)$ is the closed convex hull of the set
\begin{equation}\label{DeltaBoundary}
\Delta(Y) = \{\delta_y\mid y\in Y\}
\end{equation}
consisting of delta measures.

\begin{Def}\label{ClassCircle}
Let~$\Om$ be an open subset of~$\Mpr(Y)$ such that $\Delta(Y)\subset \Om$. 
We say that a function~$\vf\colon \mathbb{T}\to Y$ belongs to~$\ClassC(Y,\Om)$ if 
$$\cl\big\{\mu_{\vf_{\per}|_J} \mid J\subset \mathbb{R} \text{ is an interval}\big\} \subset \Om.$$
\end{Def}



To avoid technical difficulties, we will often work with simple functions. A function~$\vf \colon X\to Y$ is called simple if it attains finite number of values on a set of full measure. In other words, $\mu_{\vf}$ is a finite convex combination of delta measures. The class~$\ClassCs(Y,\Om)$ consists of simple functions in~$\ClassC(Y,\Om)$. In what follows we will often omit the symbols $Y$,~$\Om$ and write simply $\Delta, \ClassC$, and $\ClassCs$.

We will need a useful notion introduced in~\cite{SZ}.
\begin{Def}\label{Mart}
Let~$V$ be a linear space\textup, let~$\Dell \subset \Omm \subset V$. 
Let~$S = \{S_n\}_n$ be a filtration on a probability space consisting of finite algebras~$S_n$. Let also~$S_0$ be the trivial algebra. A $\Omm$-valued martingale~$M = \{M_n\}_n$ adapted to~$S$ is called a simple~$(\Omm,\Dell)$-martingale provided
\begin{enumerate}[1\textup)]
\item for any atom~$\omega \in S_n$ the convex hull of 
$\{M_{n+1}(\tilde{\omega})\}_{\tilde{\omega}\in \omega}$ belongs to~$\Omm$\textup,
\item there exists~$n_0$ such that~$M_n = M_{n_0}$ for~$n > n_0$ and~$M_{n_0} \in \Dell$.
\end{enumerate}
By $M_\infty$ we denote $M_{n_0}$ with $n_0$ as above.
\end{Def}

In this definition,~$\Omm$ and~$\Dell$ might be arbitrary sets in a vector space $V$. We will mostly use this definition for~$\Omm=\Om$,~$\Dell=\Delta$, and $V= \M$ defined above. We consider only simple martingales to avoid technical difficulties. 
Definition~\ref{Mart} slightly differs from the corresponding definition in~\cite{SZ}.
\begin{Th}\label{Main}
Let $\Om$ be an open subset of $\M(Y)$ such that $\Delta \subset \Om$. Let~$M$ be a simple~$(\Om,\Delta)$ martingale. Then\textup, there exists~$\vf \in \ClassCs$ such that~$\mu_{\vf} = M_0$.
\end{Th}
This theorem serves as one of two main ingredients in the proofs of Theorems~\ref{CoincedenceBMO},~\ref{LpNormsTheorem},~\ref{WeakTypeTheorem}, and~\ref{MuckenhouptTheorem}. It may be reduced to Lemma~\ref{Gluing} below using induction (we will provide more details at the end of this section).
\begin{Le}\label{Gluing}
Let~$\vf_0$ and~$\vf_1$ be two functions in the class~$\ClassCs$. Assume that the segment~$[\mu_{\vf_0},\mu_{\vf_1}]$ lies in~$\Om$ entirely. For any~$\alpha \in (0,1)$\textup, there exists~$\vf_{\alpha}\in\ClassCs$ such that
\begin{equation}\label{GluingIdentity}
\mu_{\vf_\alpha} = (1-\alpha)\mu_{\vf_0} + \alpha\mu_{\vf_1}.
\end{equation}
\end{Le}

The proof of Lemma~\ref{Gluing} requires some efforts. We need to introduce some notation.
\begin{Def}
Let~$\vf$ be a function on the interval~$I = [i_1,i_2]$ and let~$J$ be an interval. Define the function~$\vf_{J}$ by the rule
\begin{equation*}
\vf_{\scriptscriptstyle{J}}(x) = \vf\Big((x-j_1)\frac{i_2 - i_1}{j_2 - j_1} + i_1\Big), \quad x \in J = [j_1,j_2].
\end{equation*}
We call~$\vf_{J}$ the transfer of~$\vf$ to~$J$.
\end{Def}
This rescaling does not affect the distribution:
\begin{equation}\label{Rescaling}
\mu_{\vf} = \mu_{\vf_{J}}.
\end{equation}
\begin{Def}
Let~$\lambda \in (0,1)$. Consider the splitting of~$[-\frac12,\frac12]$ into subintervals\textup:
\begin{equation*}
I_{k,\pm} = \Big[\pm\frac{1-\lambda^{k-1}}{2}, \pm\frac{1 - \lambda^k}{2}\Big],\quad k\in \mathbb{N}.
\end{equation*}
Let~$\vf$ be a function defined on~$[-\frac12,\frac12]$. We call the function~$\Gamma_{\lambda}[\vf]$ defined on the same interval by the formula
\begin{equation*}
\Gamma_{\lambda}[\vf] = \vf_{I_{k,\pm}} \ \hbox{on the interval}\ I_{k,\pm},\quad k\in\mathbb{N},
\end{equation*}
the~$\lambda$-homogenization of~$\vf$.
In the case where~$\vf$ is initially defined on the circle\textup, its~$\lambda$-homogenization is defined as the periodic extension of the function~$\Gamma_{\lambda}[\vf_{\per}|_{[-\frac12,\frac12]}]$.
\end{Def}
Note that~$\Gamma_{\lambda}[\vf]$ has the same distribution as~$\vf$:
\begin{equation*}
\mu_{\Gamma_{\lambda}[\vf]} = \mu_{\vf}.
\end{equation*}
It is convenient to extend the splitting~$\{I_{k,\pm}\}$ periodically to the whole line. The most important property of the partition obtained is that the fraction of the lengths of any two neighbor intervals does not exceed~$\lambda$:
\begin{equation}\label{BoundedFraction}
\lambda \leq \frac{|I_{k,\pm}|}{|I_{k+1,\pm}|} \leq \lambda^{-1}.
\end{equation}
\begin{Le}\label{OneFunction}
Let~$\vf \in \ClassCs$. There exists~$\lambda_0 \in (0,1)$ depending on~$Y, \Om,$ and~$\varphi$ only and such that the function~$\Gamma_{\lambda}[\vf]$ belongs to~$\ClassCs$ for any~$\lambda \in (\lambda_0,1)$.
\end{Le}
\begin{proof}
We choose a compact set~$K \subset \Om$, which contains all the points~$\mu_{\vf_{\per}|_J}$, where~$J\subset \mathbb{R}$ runs through all the intervals. Since~$\vf$ is simple,~$K$ may be chosen finite dimensional, that is~$K \subset V$. Here~$V$ is the finite dimensional subspace of~$\M(Y)$ spanned by~$\delta_y$, where~$y$ runs through all possible values of~$\vf$. We fix the Euclidean distance  in~$V$. Since the function $\Gamma_{\lambda}[\vf]$ attains the same values as $\vf$, it suffices to show that given $\delta>0$, one can take $\lambda$ so close to 1 that 
\begin{equation}\label{Needed}
\dist(\mu_{\Gamma_{\lambda}[\vf]|_J},K) < \delta
\end{equation}
for any interval~$J \subset \mathbb{R}$. 

We will consider long and short intervals~$J$ separately. Pick~$r$ to be a positive number to be chosen later. We say that an interval~$J$ is~$r$-long if it covers at least~$r$ intervals~$I_{k,\pm}$ entirely (recall that the partition~$\{I_{k,\pm}\}$ is extended periodically to the whole line). If the interval~$J$ is not~$r$-long, we call it~$r$-short.


\paragraph{The case of~$r$-long intervals.}  Let the endpoints of~$J = [j_1,j_2]$ lie inside~$I_{k,\pm}$ and~$I_{l,\pm}$ correspondingly. Let~$\tilde{J} = J \setminus (I_{k,\pm} \cup I_{l,\pm})$. Then,
\begin{equation}\label{ConvComb}
\mu_{\Gamma_{\lambda}[\vf]|_J} = \frac{|\tilde{J}|}{|J|}\mu_{\vf} + \frac{|J\cap I_{k,\pm}|}{|J|}\mu_{\Gamma_{\lambda}[\vf]|_{J\cap I_{k,\pm}}} + \frac{|J\cap I_{l,\pm}|}{|J|}\mu_{\Gamma_{\lambda}[\vf]|_{J\cap I_{l,\pm}}}.
\end{equation}
By~\eqref{BoundedFraction},
\begin{equation*}
\frac{|J\cap I_{k,\pm}|}{|J|} +\frac{|J\cap I_{l,\pm}|}{|J|} \leq \frac{2}{1 + \sum\limits_{j=1}^{r-1}\lambda^{j}}.
\end{equation*}
Thus, equation~\eqref{ConvComb} leads to
\begin{equation}\label{Conv}
\mu_{\Gamma_{\lambda}[\vf]|_J} = \alpha_+\mu_\vf + \alpha_- y, \quad  \alpha_+ + \alpha_- = 1, \alpha_{\pm} \geq 0,
\end{equation}
where~$y$ is a point in the convex hull of~$K$ and~$\alpha_- \leq \frac{2-2\lambda}{1 - \lambda^{r}}$. The latter quantity tends to~$2/r$ when~$\lambda \to 1$. In particular,
\begin{equation}\label{FirstEstimate}
\|\mu_{\Gamma_{\lambda}[\vf]|_J} -  \mu_{\vf}\| \leq \frac{4}{r}\diam K,
\end{equation}
provided~$\lambda$ is sufficiently close to~$1$.

\paragraph{The case of~$r$-short interval.} Let~$J = [j_1,j_2]$ be an~$r$-short interval. Let it intersect~$s < r+2$ intervals of the partition~$I_{k,\pm}$ (let the intervals intersected by~$J$ be~$\{I_{k_j}\}_{j=1}^{s}$, we assume the $k_j$ to be consecutive numbers), the leftmost and rightmost intervals might be covered only partially. We represent the distribution over~$J$ as a convex combination of distributions over the partition intervals:
\begin{equation}\label{ConvComb2}
\mu_{\Gamma_{\lambda}[\vf]|_J} = \sum\limits_{j=1}^{s} \frac{|I_{k_j} \cap J|}{|J|}\mu_{\Gamma_{\lambda}[\vf]|_{I_{k_j}\cap J}}.
\end{equation}
We linearly map each of the intervals~$I_{k_j}$ onto the interval~$[k_j-\frac12,k_j+\frac12]$. Let the images of~$j_1$ and~$j_2$ be~$\tilde{j}_1$ and~$\tilde{j}_2$ and let~$\tilde{J} = [\tilde{j}_1,\tilde{j}_2]$. Then,
\begin{equation}\label{Rescaled}
\mu_{\vf_{\per}|_{\tilde{J}}} = \sum\limits_{j=1}^{s} \frac{\big|[k_j - \frac12,k_j+\frac12] \cap \tilde{J}\big|}{|\tilde{J}|}\mu_{\vf|_{[k_j - \frac12,k_j+\frac12] \cap \tilde{J}}} \stackrel{\scriptscriptstyle \eqref{Rescaling}}{=} \sum\limits_{j=1}^{s} \frac{\big|[k_j - \frac12,k_j+\frac12] \cap \tilde{J}\big|}{|\tilde{J}|}\mu_{\Gamma_{\lambda}[\vf]|_{I_{k_j}\cap J}}.
\end{equation}
Note that~\eqref{BoundedFraction} leads to
\begin{equation*}
\Big|\frac{|I_{k_j} \cap J|}{|J|} - \frac{\big|[k_j - \frac12,k_j+\frac12] \cap \tilde{J}\big|}{|\tilde{J}|}\Big| \leq \lambda^{-r-1}-1,
\end{equation*}
so, subtracting~\eqref{Rescaled} from~\eqref{ConvComb2}, we get
\begin{equation*}
\Big\|\mu_{\Gamma_{\lambda}[\vf]|_J} - \mu_{\vf_{\per}|_{\tilde{J}}}\Big\| \leq (r+2)(\lambda^{-r-1}-1)\diam K.
\end{equation*}
In particular,
\begin{equation}\label{SecondEstimate}
\dist(\mu_{\Gamma_{\lambda}[\vf]|_J}, K) \leq (r+2)(\lambda^{-r-1}-1)\diam K.
\end{equation}

\paragraph{The choice of~$r$ and~$\lambda$.} We fix~$r$ to be so large that~\eqref{FirstEstimate} leads to~\eqref{Needed} for~$\lambda$ sufficiently close to~$1$. After that we choose~$\lambda$ to be sufficiently close to one in such a manner that~\eqref{SecondEstimate} implies~\eqref{Needed}.
\end{proof}

\paragraph{End of proof of Lemma~\ref{Gluing}.} Consider~$\vf_\alpha$ given by the formula
\begin{equation*}
\vf_\alpha(x) = \begin{cases}
\Big(\Gamma_{\lambda}[\vf_1]\Big|_{[-\frac12,\frac12]}\Big)_{[0,\alpha]}(x) \quad &x \in [0,\alpha); \\
\Big(\Gamma_{\lambda}[\vf_0]\Big|_{[-\frac12,\frac12]}\Big)_{[\alpha,1]}(x) \quad &x\in [\alpha,1),
\end{cases}
\end{equation*}
and extend it periodically. Since~$\lambda$-homogenization preserves the distribution,~$\vf_\alpha$ satisfies~\eqref{GluingIdentity}. Similar to the proof of Lemma~\ref{OneFunction}, it suffices to show proper analogs of~\eqref{FirstEstimate} and~\eqref{SecondEstimate} for any interval~$J \subset \mathbb{R}$. The reasoning described in the proof of Lemma~\ref{OneFunction} works verbatim. The case of~$r$-short intervals does not differ at all (since we average only one of the functions~$\Gamma_{\lambda}[\vf_0]$ or~$\Gamma_{\lambda}[\vf_1]$ over a short interval). As for~$r$-long intervals, the only difference is that inside the intervals completely covered by~$J$ there might occur both functions~$\Gamma_{\lambda}[\vf_0]$ and~$\Gamma_{\lambda}[\vf_1]$. Therefore, the estimate~\eqref{FirstEstimate} is true with~$\mu_{\vf}$ replaced by a point inside~$[\mu_{\vf_0},\mu_{\vf_1}]$. \qed

\begin{proof}[Proof of Theorem~\ref{Main}]
With Lemma~\ref{Gluing} at hand, we proceed by induction with respect to $n_0$ (the time at which the given martingale reaches~$\Delta$). The base $n_0=0$ is trivial. To make the induction step, we use the induction hypothesis for restriction of the martingale $M$ to each atom of $S_1$ and construct the corresponding functions from $\ClassCs$. Consecutive application of Lemma~\ref{Gluing} allows to glue the desired function $\vf \in \ClassCs$ with~$\mu_{\vf} = M_0$ from these pieces.
\end{proof}

\section{Proofs of Theorems~\ref{LpNormsTheorem},~\ref{WeakTypeTheorem}, and~\ref{MuckenhouptTheorem}}\label{S3}

We need to introduce some terminology from~\cite{SZ}. 



Let~$\Xi_0$ be a non-empty open convex subset of~$\mathbb{R}^2$ that does not contain lines. Let~$\Xi_1$ be another open convex subset of~$\mathbb{R}^2$ such that~$\cl \Xi_1 \subset \Xi_0$. We define the set~$\Omega$ as~$\cl (\Xi_0 \setminus  \Xi_1)$ and the class~$\Class = \Class(\Omega)$ of integrable~$\mathbb{R}^2$-valued functions on an interval~$I \subset \mathbb{R}$:
\begin{equation}\label{AnalyticClass}
\Class = \big\{\vf \in L_1(I,\mathbb{R}^2)\,\big|\,\, \vf(I) \subset \partial \Xi_0, \,\,\forall J \hbox{---subinterval of}\,\,I\quad \av{\vf}{J} \notin \Xi_1\big\}.
\end{equation}
Let~$f$ be a bounded from below Borel measurable locally bounded function on~$\partial \Xi_0$. Consider the Bellman function~$\Bell = \Bell(\Omega,f)$:
\begin{equation}\label{BellmanFunction}
\Bell(x) = \sup\big\{\av{f(\vf)}{I}\,\big|\,\, \av{\vf}{I} = x,\,\,\vf \in \Class\big\}, \quad x \in \Omega.
\end{equation}
\begin{Th}[Main Theorem in~\cite{SZ}]\label{OldMain}
Assume 
\begin{enumerate}[1\textup)]
\item the sets~$\Xi_1$ and~$\Xi_0$ are strictly convex\textup;
\item the boundary of~$\Xi_1$ is~$C^2$-smooth\textup;
\item the maximal inscribed cones of~$\Xi_1$ and~$\Xi_0$ coincide.
\end{enumerate}
Then\textup,~$\Bell$ is the pointwise minimal among locally concave on~$\Omega$ functions~$G$ that satisfy~$G(x) \geq f(x)$\textup,~$x\in \partial \Xi_0$.
\end{Th}
By the maximal inscribed cone of a convex set $\Xi_j$, we mean the maximal by inclusion convex cone contained in $\Xi_j - \Xi_j$, $j=0,1$. We call a function~$G\colon \Omega \to \mathbb{R}\cup \{\pm\infty\}$ locally concave provided its restriction to any convex subset of~$\Omega$ is concave. In the following lemma we use Definition~\ref{Mart} with $\Dell:= \partial \Xi_0$, $\Omm := \Omega$ and $V:=\mathbb{R}^2$.
\begin{Le}\label{AntiBellman}
For any~$x\in \Omega$ and any~$\theta > 0$\textup, there exists a simple~$(\Omega,\partial \Xi_0)$-martingale~$M$ such that
\begin{equation}\label{OptimalMartingale}
\E f(M_{\infty}) + \theta \geq \Bell(x), \quad M_0 = x.
\end{equation}
\end{Le}
\begin{proof}
Consider the function
\begin{equation*}
B(x) = \sup\Big\{\E f(M_{\infty})\;\Big|\, M\hbox{ is a simple ($\Omega,\partial \Xi_0$)-martingale, } M_0 = x\Big\}, \quad x \in \Omega.
\end{equation*}
This function is locally concave on~$\Omega$ (see the proof of Lemma~$2.17$ in~\cite{SZ}) and $B(x) = f(x)$ for $x \in \partial \Xi_0$; note that $B \ne -\infty$. On the other hand, for any locally concave function~$G$, the sequence~$\E G(M_n)$ is non-increasing with~$n$ (see Lemma~$2.10$ in~\cite{SZ}), so
\begin{equation*}
B(x) \leq G(x),\quad \text{provided $G$ is locally concave on~$\Omega$ and } G(x) \geq f(x), \ x\in\partial \Xi_0.
\end{equation*}
So~$B$ coincides with~$\Bell$ by Theorem~\ref{OldMain}. The existence of the desired martingale~$M$ is now guaranteed by the definition of~$B$.
\end{proof}

Let now~$\hat{\Xi}$ be a set that satisfies the same assumptions as~$\Xi_1$, but lies strictly inside the former set:~$\cl\hat{\Xi} \subset \Xi_1$. 
Consider the class~$\ClassCs(\partial\Xi_0, \hat\Om(\Xi_0,\hat\Xi))$ (see  Definition~\ref{ClassCircle}) generated by 
\begin{equation}\label{Star}
\hat{\Om}(\Xi_0,\hat\Xi) = \Big\{\mu \in \Mpr(\partial\Xi_0)\;\Big|\,\int\limits_{\partial \Xi_0} xd\mu(x) \notin \cl\hat{\Xi}\Big\}.
\end{equation}
\begin{Le}\label{Transference}
For any~$\Omega$\textup, any~$\hat{\Xi}$\textup, any~$x\in \Omega$\textup, and any $\theta>0$\textup, there exists~$\vf \in \ClassCs(\partial \Xi_0, \hat{\Om}(\Xi_0,\hat\Xi))$ such that~$\av{\vf}{\mathbb{T}} = x$ and
\begin{equation*}
\av{f(\vf)}{\mathbb{T}} + \theta \geq \Bell(x).
\end{equation*}
\end{Le}
\begin{proof}
Let~$M$ be a simple~$(\Omega,\partial \Xi_0)$-martingale such that~\eqref{OptimalMartingale} holds true, such a martingale exists by Lemma~\ref{AntiBellman}. There exists a unique~$\Mpr(\partial \Xi_0)$-valued martingale~$\mathbb{M}$ over the same filtration such that 
\begin{equation*}
\int\limits_{\partial \Xi_0} x d\mathbb{M}_n(x) = M_n \quad \hbox{for any $n$}.
\end{equation*}
It is easy to see that~$\mathbb{M}$ is a simple~$(\hat{\Om},\Delta)$-martingale and
\begin{equation*}
\int\limits_{\partial \Xi_0} f(x) \,d\mathbb{M}_0(x) = \E f(M_{\infty}).
\end{equation*}
It remains to apply Theorem~\ref{Main} for~$\mathbb{M}$ and construct the desired function $\vf$.
\end{proof}

Consider the class~$\ClassC (\mathbb{R}, \Om^{p,\eps})$ generated by the set
\begin{equation}\label{BMOpClass}
\Om^{p,\eps} = \Big\{\mu \in \Mpr(\mathbb{R})\;\Big|\,\int\limits_\mathbb{R}\Big|t - \int\limits_\mathbb{R}\tau\,d\mu(\tau)\Big|^p\,d\mu(t) < \eps^p \Big\}.
\end{equation}
This class is closely related to the~$\BMO_p$-norm:
\begin{align}\label{SimpleEmbeddings}
\|\vf\|_{p,\mathbb{T}} < \eps \quad \Rightarrow \quad \vf \in \ClassC(\mathbb{R}, \Om^{p,\eps});\\
\|\vf\|_{p,\mathbb{T}} \leq \eps \quad \Leftarrow \quad \vf \in \ClassC(\mathbb{R}, \Om^{p,\eps}). \label{eq150401}
\end{align}
The case~$p=2$ is special. The class~$\ClassC(\mathbb{R}, \Om^{2,\eps})$ admits an alternative definition in the style of~\eqref{Star}:
\begin{equation}\label{eq210701}
\vf  \in \ClassC(\mathbb{R}, \Om^{2,\eps}) \Longleftrightarrow (\vf,\vf^2) \in \ClassC(\partial\Xi_0, \hat\Om(\Xi_0,\Xi_\eps)),
\end{equation}
where $\Xi_0 = \{(x_1,x_2)\in \mathbb{R}^2\mid x_2>x_1^2\}$ and $\Xi_\eps = \{(x_1,x_2)\in \mathbb{R}^2\mid x_2>x_1^2+\eps^2\}$.



\begin{proof}[Proof of Theorem~\ref{LpNormsTheorem}.]
By Remark~\ref{TwoConstants}, it suffices to construct a function~$\vf_1 \colon \mathbb{T} \to\mathbb{R}$ such that
\begin{equation*}
\av{\vf_1}{\mathbb{T}} = 0,\quad \av{|\vf_1|^p}{\mathbb{T}} +\theta \geq \frac{p}{2}\Gamma(p), \quad \hbox{and}\quad \|\vf_1\|_{2,\mathbb{T}} \leq 1+\delta,
\end{equation*}
given any positive~$\theta,\delta$, and $p > 2$. We choose
\begin{equation*}
\begin{split}
\Xi_0 =\{(x_1,x_2)\in\mathbb{R}^2\mid x_2 > x_1^2\},&\qquad
\Xi_1 = \{(x_1,x_2)\in\mathbb{R}^2\mid x_2 > x_1^2+1\}, \\
f(x_1,x_1^2) &= |x_1|^p, \qquad x_1 \in \mathbb{R},
\end{split}
\end{equation*}
and consider the corresponding function~$\Bell$ defined by formula~\eqref{BellmanFunction} on $\Omega = \cl\Xi_0\setminus\Xi_1$. The exact formula for this function was computed in~\cite{SlavinVasyuninLp}. We need the value at a certain point only:
\begin{equation*}
\Bell(0,1) = \frac{p}{2}\Gamma(p).
\end{equation*}
We use~Lemma~\ref{Transference} with 
$$
\hat\Xi = \Xi_{1+\delta} = \{x\in\mathbb{R}^2\mid x_2 > x_1^2 + (1+\delta)^2\}
$$
and obtain a function $\vf=(\vf_1,\vf_1^2) \in \ClassCs(\partial \Xi_0, \hat{\Om}(\Xi_0,\Xi_{1+\delta}))$ with  
$$
\av{\vf}{\mathbb{T}} = (0,1), \quad
\av{|\vf_1|^p}{\mathbb{T}} + \theta \geq \frac{p}{2}\Gamma(p).
$$
In view of~\eqref{eq210701} and~\eqref{eq150401} with $\eps = 1+\delta$, we have~$\|\vf_1\|_{2,\mathbb{T}} \leq 1+\delta$.
\end{proof}

\begin{proof}[Proof of Theorem~\ref{WeakTypeTheorem}.]
By Remark~\ref{TwoConstants}, it suffices to prove the sharpness of the second and the third inequalities in~\eqref{JohnNirenbergQuadraticNorm} for functions on the line. We choose the same~$\Omega$ and $\hat\Xi$ as in the proof of Theorem~\ref{LpNormsTheorem} above and another~$f$:
\begin{equation*}
f(x_1,x_1^2) =\chi_{\scriptscriptstyle[\lambda,+\infty)}(|x_1|),\qquad x_1\in \mathbb{R}.
\end{equation*}
The exact formula for the corresponding Bellman function was computed in~\cite{VasyuninVolberg}. In particular,
\begin{equation*}
\Bell(0,1) = \begin{cases}
\frac{1}{\lambda^2} ,\quad \lambda \in [1,2];\\
\frac{e^2}{4}e^{-\lambda},\quad \lambda \geq 2.
\end{cases}
\end{equation*}
Similar to the previous proof, for any given positive $\delta,\theta$, Lemma~\ref{Transference} (with~\eqref{eq150401} and~\eqref{eq210701}) allows us to construct a function~$\vf_1\colon \mathbb{T} \to \mathbb{R}$ such that
\begin{equation*}
\av{\vf_1}{\mathbb{T}} = 0, \ \|\vf_1\|_{2,\mathbb{T}} \leq 1 + \delta, \quad \hbox{and} \quad \Big|\Big\{t\in\mathbb{T}\;\Big|\, |\vf_1(t)| \geq \lambda\Big\}\Big| +\theta\geq  \begin{cases}
\frac{1}{\lambda^2} ,\quad \lambda \in [1,2];\\
\frac{e^2}{4}e^{-\lambda},\quad \lambda \geq 2.
\end{cases}
\end{equation*}
\end{proof}
\begin{proof}[Proof of Theorem~\ref{MuckenhouptTheorem}.]
Consider the class~$\ClassC(\mathbb{R}_+, \Ommuc^{p,C})$ generated by the set
\begin{equation}\label{MuckenhouptGeneral}
\Ommuc^{p,C} = \Big\{\mu \in \Mpr(\mathbb{R}_+)\;\Big|\,\Big(\int\limits_{\mathbb{R}_+}t d\mu(t)\Big)\cdot\Big(\int\limits_{\mathbb{R}_+}t^{-\frac{1}{p-1}}\,d\mu(t)\Big)^{p-1} < C\Big\}
\end{equation}
(see Definition~\ref{ClassCircle}). This class is closely related to the~$A_p$-constant:
\begin{align}
[w]_{A_p(\mathbb{T})} < C \quad \Rightarrow \quad w \in \ClassC(\mathbb{R}_+, \Ommuc^{p,C});\notag \\
[w]_{A_p(\mathbb{T})} \leq C \quad \Leftarrow \quad w \in \ClassC(\mathbb{R}_+, \Ommuc^{p,C}).\label{eq210703}
\end{align}

The class~$\ClassC(\mathbb{R}_+, \Ommuc^{p,C})$ might be described in the style of~\eqref{Star}:
\begin{equation}\label{eq210704}
w_1  \in \ClassC(\mathbb{R}_+, \Ommuc^{p,C}) \Longleftrightarrow (w_1,w_1^{-\frac{1}{p-1}}) \in \ClassC(\partial\Xi_0, \hat\Om(\Xi_0,\Xi_C)),
\end{equation}
where 
\begin{equation}\label{eq210705}
\begin{split}
\Xi_0 &= \{(x_1,x_2) \in \mathbb{R}_+^2\mid x_2>x_1^{-\frac{1}{p-1}}\},\\
\Xi_C &= \{(x_1,x_2) \in \mathbb{R}_+^2\mid x_2> C x_1^{-\frac{1}{p-1}}\}, \quad \text{ for } C>1.
\end{split}
\end{equation}


Let $R(C,q)$ be equal to the value of the right-hand side of~\eqref{eq210702}. It follows from~\cite{VasyuninMuckenhoupt} and~\cite{Vasyunin3} that 
$$
R(C,q) = \sup_{\tilde C < C} R(\tilde C,q).
$$ 
In order to prove~\eqref{eq210702} it suffices, for any given $\theta>0$ and $\tilde C<C$ to find a weight $w_1$ with 
$$
[w_1]_{A_p(\mathbb{\mathbb{T}})} \leq C,\qquad \frac{\av{w_1^q}{\mathbb{T}}^{\frac{1}{q}}}{\av{w_1}{\mathbb{T}}} + \theta \geq R(\tilde C,q).
$$

Fix any $\tilde C<C$ and take $\Xi_0, \Xi_{\tilde C}$ as in~\eqref{eq210705}, and $f(x_1, x_1^{-\frac{1}{p-1}}) = x_1^q, \quad x_1>0$.
We consider the corresponding function~$\Bell$ defined by formula~\eqref{BellmanFunction} on $\Omega = \cl \Xi_0\setminus\Xi_{\tilde{C}}$. Then,
$$
R^q(\tilde C,q) = \max\{\Bell(1,x_2)\mid 1\leq x_2\leq \tilde C\}.
$$

We fix any $\theta>0$ and apply Lemma~\ref{Transference} with
$$
\hat\Xi = \Xi_C = \{(x_1,x_2) \in \mathbb{R}_+^2\mid x_2>{C}x_1^{-\frac{1}{p-1}}\}
$$
and find a function $w = (w_1, w_1^{-\frac{1}{p-1}}) \in \ClassC(\partial\Xi_0, \hat\Om(\Xi_0,\Xi_C))$ such that 
$$
\av{w_1}{\mathbb{T}}=1, \qquad \av{w_1^q}{\mathbb{T}} + \theta \geq R^q(\tilde C,q).
$$
Now,~\eqref{eq210703} and~\eqref{eq210704} imply that $[w_1]_{A_p(\mathbb{T})} \leq C$.

%


\end{proof}
\begin{Rem}
One can use the same machinery to transfer weak Reverse H\"older inequalities. The corresponding sharp constants were computed in~\textup{\cite{Reznikov}}.
\end{Rem}

\section{Proof of Theorem~\ref{CoincedenceBMO}}\label{S4}
It is proved in~\cite{Slavin} and~\cite{SlavinVasyunin} that the optimal constant~$C_{3,p}^{I}$ defined by~\eqref{C3pI}, is obtained at the function
\begin{equation*}
\vf(x) = \log x, \quad x\in [0,1].
\end{equation*}
So, to prove Theorem~\ref{CoincedenceBMO}, it suffices to find a simple~$(\Om,\Delta)$-martingale (with~$\Om =\Om^{p,\eps} $ as in~\eqref{BMOpClass} with~$\eps = \|\log x\|_{p,[0,1]}$) such that~$M_0 = \mu_{\log x}$ (here and in what follows $\mu$ stands for the Lebesgue measure). Seemingly, such a martingale does not exist. We will construct simple martingales for which this relation is almost fulfilled, and then pass to the limit.

For real $\lambda>1$ and natural $k$ let
$
I_k = [\lambda^{-k},\lambda^{-k+1}], \tilde{I}_k = [0,\lambda^{-k}].
$

Let~$N$ be natural. Consider the splitting of~$[0,1]$ into the intervals~$I_k$, $k=1,2,\ldots, N$, and $\tilde{I}_N$. 
Define the function~$\vf_{\lambda,N}$ by the rule
\begin{equation}\label{Approximation}
\vf_{\lambda,N}(x) = \begin{cases}
\frac{1}{|I_{k}|}\int\limits_{I_k} \log x\,dx,\quad x\in I_k,\quad k=1,2,\ldots, N;\\
-N\log \lambda,\quad x\in \tilde{I}_N.
\end{cases}
\end{equation}
Let~$S_n$ be the algebra on $[0,1]$ generated by the intervals~$I_1,I_2,\ldots,I_n$, here~$n=1,2,\ldots,N$. 
Consider~$M$ given by the formula
\begin{equation*}
M_n  =\sum\limits_{k \leq n}\delta_{\vf_{\lambda,N}|_{I_k}}\chi_{I_k} +  \mu_{\vf_{\lambda,N}|_{\tilde{I}_n}}\chi_{\tilde{I}_n}, \quad n = 0,1,\ldots, N,
\end{equation*}
and~$M_n = M_N$ for~$n \geq N$. Then~$(M_n,S_n)$ is a simple martingale that starts from~$M_0 = \mu_{\vf_{\lambda,N}}$, which is a convex combination of~$N$ delta measures. On each step, the martingale cuts one delta measure off this convex combination.
\begin{Le}\label{ConstructionOfMartingale}
For any~$\delta > 0$\textup, there exists~$\lambda_0>1$ such that for any~$\lambda \in (1, \lambda_0)$ and any~$N$\textup, the martingale~$M$ is an~$(\Om^{p,\eps},\Delta(\mathbb{R}))$-martingale\textup, where~$\Om^{p,\eps}$ is given by~\eqref{BMOpClass} with~$\eps = \|\log x\|_{p,[0,1]} + \delta$.
\end{Le}
\begin{proof}
By Definition~\ref{Mart}, we need to show that the segments~$[\mu_{\vf_{\lambda,N}|_{\tilde{I}_n}},\delta_{\vf_{\lambda,N}|_{I_n}}]$ lie in the domain~$\Om^{p,\eps}$. 
For that we consider yet another function on $\mathbb{R}_+$:
\begin{equation*}
\psi_{\lambda,n}(x) = \begin{cases}
\vf_{\lambda,N}(x), \quad 0 < x \leq \lambda^{-n};\\
\frac{1}{|I_{n}|}\int\limits_{I_n} \log x\,dx, \quad x > \lambda^{-n}.
\end{cases}
\end{equation*}
This function satisfies the bound
\begin{equation*}
\|\psi_{\lambda,n}\|_{p,[0,\infty)} \leq \|\log x\|_{p,[0,\infty)} + |\log \lambda\,| = \|\log x\|_{p,[0,1]} + |\log \lambda\,|,
\end{equation*}
here we have used the fact that truncation does not increase the $\BMO_p$ norm (see~Appendix~\ref{AppA}).

So, take~$\lambda_0 = e^{-\frac15 \delta}$. Then, for any $\lambda \in (1,\lambda_0)$ we have
\begin{equation*}
\|\psi_{\lambda,n}\|_{p,[0,\infty)}<\|\log x\|_{p,[0,1]}+\delta.
\end{equation*}
Any point inside the interval~$(\mu_{\vf_{\lambda,N}|_{\tilde{I}_n}},\delta_{\vf_{\lambda,N}|_{I_n}})$ may be realized as~$\mu_{\psi_{\lambda,n}|_{[0,s]}}$ for some~$s\in [\lambda^{-n},\infty)$, therefore it lies in $\Om^{p,\eps}$.
\end{proof}
\begin{proof}[Proof of Theorem~\ref{CoincedenceBMO}.] By the results of~\cite{Slavin} and~\cite{SlavinVasyunin}, 
\begin{equation*}
\int\limits_0^1 e^{C_{3,p}^I \frac{\log x}{\|\log x\|_{p,[0,1]}}} = \infty.
\end{equation*} 
Since~$C_{3,p}^{\mathbb{R}} \leq C_{3,p}^I$, it suffices, given any~$m \in \mathbb{N}$, to construct a function~$\vf_m$ on the line such that
\begin{equation*}
\int\limits_0^1 e^{C_{3,p}^I \frac{\vf_m}{\|\vf_m\|_{p,\mathbb{R}}}} > m.
\end{equation*}
The function~$\vf_m$ will be periodic. We consider the functions~$\vf_{\lambda,N}$ constructed above and notice that there exists~$\lambda$ and~$N$ such that
\begin{equation*}
\int\limits_0^1 e^{C_{3,p}^I \frac{\vf_{\lambda,N}}{\|\log x\|_{p,[0,1]}}} > m.
\end{equation*}
So, it suffices to construct a function on the circle that has the same distribution as~$\vf_{\lambda,N}$ and whose~$\BMO_p(\mathbb{T})$-norm is arbitrarily close to~$\|\log x\|_{\BMO_p([0,1])}$. Such a function is provided by Lemma~\ref{ConstructionOfMartingale} and Theorem~\ref{Main}.
\end{proof}


\appendix
\section{Truncations in $\BMO_p$}\label{AppA}
The facts surveyed in this section seem to be a part of folklore. The authors did not manage to find the proofs in the literature and provide them for completeness.
\begin{Le}
For any~$\varphi\colon I \to \mathbb{R}$ and any non-decreasing function~$g\colon \mathbb{R}\to\mathbb{R}$\textup,
\begin{equation*}
\|g(\varphi)\|_{p,I} \leq \|g\|_{\mathrm{Lip}(\mathbb{R})}\|\varphi\|_{p,I}.
\end{equation*}
\end{Le}
By the Lipschitz constant of a function we mean
\begin{equation*}
\|g\|_{\mathrm{Lip}} = \sup\limits_{x,y\in\mathbb{R}}\frac{|g(x)-g(y)|}{|x-y|}.
\end{equation*} 
\begin{proof}
Without loss of generality,~$I = [0,1]$, and~$g$ is one-Lipschitz. It suffices to prove the estimate
\begin{equation}\label{eq290701}
\int\limits_0^1 |g(\varphi(t)) - \av{g(\varphi)}{[0,1]}|^p\,dt \leq \|\varphi\|_{p,I}^p.
\end{equation}
Since the monotonic rearrangement does not increase the $\BMO_p$-norm (see~\cite{Klemes} and \cite{Shanin2}), one may assume~$\vf$ to be non-decreasing. Consider the function~$\psi = g(\vf)+c$ such that $\int_0^1 \psi = \int_0^1 \vf$ and $c$ is a constant. The function 
$$
I(t) = \int\limits_0^t \vf - \int\limits_0^t \psi, \quad t \in [0,1],
$$
is convex because $I'=\vf-\psi$ does not decrease. Moreover, $I(0)=I(1)=0$. Therefore, $I \leq 0$ on $[0,1]$. By Karamata's inequality, 
$$
\int\limits_0^1 |\vf - \av{\vf}{[0,1]}|^p \geq \int\limits_0^1 |\psi - \av{\psi}{[0,1]}|^p,
$$ 
which implies~\eqref{eq290701}.
\end{proof}

\begin{Cor}
For any real~$N,$ the truncated function~$\min(\varphi,N)$ has the same or smaller~$\BMO_p$ norm than~$\varphi$.
\end{Cor}

D. Stolyarov

\medskip

Department of Mathematics and Computer Sciences, St. Petersburg State University, Russia

St. Petersburg Department of Steklov Mathematical Institute, Russia

\medskip

d.m.stolyarov@spbu.ru

\bigskip

P. Zatitskiy

\medskip

Department of Mathematics and Computer Sciences, St. Petersburg State University, Russia

St. Petersburg Department of Steklov Mathematical Institute, Russia

\medskip

pavelz@pdmi.ras.ru

\end{document}